\documentclass{amsart}

\title{Uniqueness of $BP \lan{n}$}
\author{Vigleik Angeltveit}
\author{John A. Lind}

\usepackage{amsxtra}
\usepackage{amsfonts}
\usepackage[latin1]{inputenc}
\usepackage{graphicx}
\usepackage{amsmath,amssymb,latexsym,amsthm,mathrsfs}
\usepackage[all]{xy}
\usepackage{pstricks}

\newtheorem{theorem}{Theorem}[section]

\newtheorem{lemma}[theorem]{Lemma}

\newtheorem{defn}[theorem]{Definition}

\theoremstyle{definition}
\newtheorem{remark}[theorem]{Remark}

\makeatletter
\let\c@equation\c@theorem
\makeatother
\numberwithin{equation}{section}

\newtheorem{lettertheorem}{Theorem}

\newcommand{\cA}{\mathcal{A}}

  \newcommand{\bC}{\mathbb{C}}   \newcommand{\bF}{\mathbb{F}}          
         \newcommand{\bZ}{\mathbb{Z}}

\newcommand{\ttau}{\bar{\tau}}
\newcommand{\xxi}{\bar{\xi}}
\newcommand{\lan}[1]{\langle {#1} \rangle}

\providecommand{\abs}[1]{\lvert #1 \rvert}
\providecommand{\pc}{^{\wedge}_{p}}

\begin{document}
 
\begin{abstract}
Fix a prime number $p$ and an integer $n \geq 0$. We prove that if a $p$-complete spectrum $X$ satisfying a mild finiteness condition has the same mod $p$ cohomology as $BP \lan{n}$ as a module over the Steenrod algebra, then $X$ is weak homotopy equivalent to the $p$-completion of $BP\lan{n}$.
\end{abstract}

\maketitle

\section{Introduction}
Let $MU$ be the spectrum representing complex cobordism, and recall that when localized at a prime $p$ there is an idempotent map $\epsilon : MU_{(p)} \to MU_{(p)}$ with image the Brown-Peterson spectrum $BP$ \cite{Qu69}. The coefficient ring of $BP$ is
\[
 \pi_* BP = \bZ_{(p)}[v_1,v_2,\ldots]
\]
with $\abs{v_k}=2 p^k - 2$. For each $n \geq 0$ we can quotient out by the generators $v_k$ for $k > n$ and construct a spectrum $BP \lan{n}$ with
\[
 \pi_* BP \lan{n} = \bZ_{(p)}[v_1,\ldots,v_n].
\]
The construction of $BP \lan{n}$ can be carried out using the Baas-Sullivan bordism theory of manifolds with singularity or by taking the cofiber of the multiplication by $v_k$ map.  It is customary to let $v_0 = p$ and to let $BP \lan{-1}=H\bF_p$. Special cases include $BP \lan{0}=H\bZ_{(p)}$ and $BP\lan{1}=\ell$, the Adams summand of connective complex $K$-theory localized at $p$. We think of $BP \lan{n}$ for varying $n$ as interpolating between mod $p$ cohomology and the Brown-Peterson summand of $p$-local complex cobordism.

Complex cobordism has played a central role in algebraic topology because of its connection to formal groups. The ring $MU^*(\bC P^\infty)$ carries a formal group and Quillen  \cite{Qu69} showed that it is the \emph{universal} formal group. Similarly, $BP$ determines the universal $p$-typical formal group and $BP \lan{n}$ is closely related to height $n$ formal groups. By further quotienting $BP\lan{n}$ out by $v_k$ for $k < n$ and inverting $v_n$ we get the $n$-th Morava $K$-theory spectrum $K(n)$, which represents the universal height $n$ formal group.

While the homotopy type of $BP$ is well defined because it comes from the Quillen idempotent on $MU$, the classes $v_k$ are not quite canonical. Indeed, there are two popular choices of $v_k$ given by Hazewinkel \cite{Ha76} and Araki \cite{Ar73}, and these are not the only ones. Hence it is not at all obvious that the homotopy type of $BP \lan{n}$ is well defined; the definition of $BP \lan{n}$ appears to depend on the choice of $v_k$ for $k > n$.

In this paper we show that the $p$-adic homotopy type of $BP \lan{n}$ is well defined by showing that it is determined by its mod $p$ cohomology.  To be more precise, fix a model of the spectrum $BP \lan{n}$ constructed using any choice for the generators $v_k$, and let $BP\lan{n}\pc$ denote its $p$-completion.  Recall that $p$-completion leaves mod $p$ cohomology unchanged and that for spectra of finite type, such as $BP\lan{n}$, $p$-completion has the effect of $p$-adically completing all homotopy groups \cite{Bo79}.  Our main result states that any $p$-complete spectrum satisfying standard finiteness conditions which has the same mod $p$ cohomology as $BP \lan{n}$ as a module over the Steenrod algebra is equivalent to $BP \lan{n}\pc$.

\begin{lettertheorem} \label{t:main}
Suppose $X$ is a spectrum which is bounded below, and whose homotopy groups are finitely generated over $\bZ_{p}$. Suppose also that there is an isomorphism
\[ \theta : H^*(BP\lan{n}; \bF_p) \to H^*(X; \bF_p) \]
of modules over the Steenrod algebra. Then there exists a weak homotopy equivalence $f : X \to BP\lan{n}_p^{\wedge}$ of spectra which induces the isomorphism $\theta$ in mod $p$ cohomology.
\end{lettertheorem}

\noindent The mod $p$ cohomology of $BP\lan{n}$ does not depend on the choices made in the construction of $BP\lan{n}$ \cite[Prop.\ 1.7]{Wi75}, so we deduce that the $p$-adic homotopy type of $BP\lan{n}$ does not either.

We believe that a stronger form of this theorem is true, where $\bZ_p$ is replaced by the $p$-local integers and the conclusion is that $X$ is equivalent to $BP\lan{n}$.  The only obstacle is a question of convergence of the Adams spectral sequence, see Remark \ref{remark:convergence}.

One motivation for proving this result is our interest in twisted cohomology theories, and in particular twisted $BP \lan{n}$-cohomology. The case $p=2$ and $n=2$ is of particular interest. In \cite{Wi73, Wi75} Wilson showed that up to localization at $2$ we have
\[ \Omega^\infty BP \lan{2} \simeq \bZ \times K(\bZ,2) \times BSU \times B^3 SU \times BP \lan{2}_{12}, \]
where $BP \lan{2}_{12}$ is the twelfth space in the $\Omega$-spectrum for $BP \lan{2}$. Hence
\[ GL_1(BP \lan{2}) \simeq \bZ/2 \times K(\bZ,2) \times BSU \times B^3 SU \times BP \lan{2}_{12}.\]
By \cite{LaNa11}, $BP \lan{2}$ is an $E_\infty$ ring spectrum at $p=2$. Hence $GL_1(BP \lan{2})$ is an infinite loop space, and its classifying space $BGL_1(BP \lan{2})$ and the associated spectrum $bgl_1(BP \lan{2})$ of units are both defined. We conjecture that the above splitting deloops, so that
\begin{equation}\label{eq:conj_splitting} 
BGL_1 BP \lan{2} \simeq K(\bZ/2, 1) \times K(\bZ, 3) \times B^2 SU \times B^4 SU \times BP \lan{2}_{13}.
\end{equation}
If the splitting \eqref{eq:conj_splitting} exists, then, just as integral cohomology of $X$ can be twisted by a class in $H^1(X, \bZ/2)$ and complex $K$-theory can be twisted by a class in $H^3(X; \bZ)$, the $BP\lan{2}$-cohomology of $X$ can be twisted by a class in
\[ ku^7(X) = [X, B^4 SU]. \] 
This has a nice interpretation in terms of the ``Bockstein'' spectral sequence
\[ E_2^{*,*} = ku^*(X)[v_2] \Rightarrow BP\lan{2}^*(X),\]
with the first nontrivial differential being modified by the twisting. We hope to return to this elsewhere.

\subsection{Main proof idea}
The case $n=1$ of Theorem \ref{t:main} is a result of Adams and Priddy \cite{AdPr76}, and our proof is modeled on their proof, with a few enhancements.  Let $H^*(-)$ denote mod $p$ cohomology, always considered as a module over the Steenrod algebra $\cA$. The proof uses the Adams spectral sequence
\[ E_2^{s,t} = Ext_\cA^{s,t}(H^*(Y), H^*(X)) \Longrightarrow \pi_{t-s} Hom(X,Y^\wedge_p).\]
The isomorphism $\theta \colon H^*(Y) \to H^*(X)$ represents an element $[\theta] \in E_2^{0,0}$.  If $[\theta]$ is a permanent cycle then it represents a map $f^{\wedge}_{p} \colon X \to Y^{\wedge}_{p}$ whose induced map on cohomology is $\theta$. We will implement this plan for $Y = BP\lan{n}$.

The mod $p$ cohomology of $BP \lan{n}$ is given by
\[ H^*(BP\lan{n}) = \cA//E_n, \]
where $\cA$ is the mod $p$ Steenrod algebra and $E_n=E(Q_0,\ldots,Q_n)$ is the exterior algebra on the first $n + 1$ Milnor primitives. We use a change-of-rings theorem to compute the $E_2$ term of the Adams spectral sequence:
\[ E_2^{s,t} = Ext_A^{s,t}(H^*(BP\lan{n}), H^*(X)) \cong Ext_{E_n}^{s,t}(\bF_p, H^*(X)).\]
We are left with $Ext$ over an exterior algebra on $n + 1$ generators, which is much more computable than $Ext$ over the full Steenrod algebra.

We will prove Theorem \ref{t:main} by showing that all the possible differentials on $[\theta]$ land in trivial groups. In other words, we prove that $E_2^{s,t}=0$ for all $(s,t)$ with $s \geq 2$ and $t-s=-1$. (In fact $E_2^{s,t}=0$ also for $(s,t)=(0,-1)$ and $(s,t)=(1,0)$.) This was also the strategy in Adams and Priddy's paper \cite{AdPr76}. In our case the required $Ext$ calculation is more difficult; in essence we have to understand an $(n+1)$-cube of ``commuting'' spectral sequences. In \S\ref{s:genSS}, we make some general observations about the kind of spectral sequences we need to analyze.

\subsection{Acknowledgements}
The authors would like to thank Craig Westerland for many helpful conversations. Without him we would not have started thinking about twisted $BP \lan{n}$-cohomology and we would not have been led to the main result of this paper. The first author was supported by an ARC Discovery grant. The second author was partially supported by the DFG through SFB1085, and thanks the Australian National University for hosting him while this research was conducted.

\section{A cube of spectral sequences} \label{s:genSS}
Suppose that we are given an $(n+1)$-dimensional chain complex $M=M^{*,\ldots,*}$, and that we are interested in the homology of the total complex. In other words, $M$ is an abelian group with $n + 1$ different gradings and we have differentials $d^0,\ldots,d^n$ on $M$, where $d^i$ increases the $i$'th grading of $M$ by $1$. The differentials are required to commute in the graded sense, meaning that $d^i d^j = -d^j d^i$, and we wish to compute the homology of $M$ with respect to the differential $d=d^0+\dotsm+d^n$.

Given a $2$-dimensional chain complex, i.e., a double complex, there are standard spectral sequences
\[ E_2^{s,t} = H^s(H^t(M,d^0), d^1) \Rightarrow H^{s+t}(M, d) \]
and
\[ E_2^{s,t} = H^s(H^t(M,d^1),d^0) \Rightarrow H^{s+t}(M,d).\]

In the situation where we have $n+1$ directions, we get more spectral sequences. We can first take the homology in the $d^n$-direction, then run the $d^{n-1}$ spectral sequence to compute $H_*(M,d^{n-1}+d^n)$. Next we have a $d^{n-2}$ spectral sequence, which takes the form
\[ H^s(H^t(M, d^{n-1}+d^n), d^{n-2}) \Rightarrow H^{s+t}(M,d^{n-2}+d^{n-1}+d^n),\]
and so on until the final $d^0$ spectral sequence. If we can compute all the differentials and solve all the extension problems this tells us $H_*(M, d)$. Or we can choose to take homology with respect to the $d^i$'s in a different order.

We can organize all of this in an $(n+1)$-cube with $M$ at the initial vertex and $H_*(M,d)$ at the terminal vertex. At the vertex corresponding to $I \subset \{0,1,\ldots,n\}$ we have $H^*(M, d^I)$, where $d^I=\sum_{i \in I} d^i$. The $i$'th edge originating at the initial vertex corresponds to taking homology with respect to $d^i$, all the other edges correspond to spectral sequences. For each $j \not \in I$ we have a spectral sequence
\[
 H^s(H^t(M, d^I), d^j) \Rightarrow H^{s+t}(M, d^{I \cup j}).
\]

For $n=2$ the cube of spectral sequences looks as follows.

\[
 \xymatrix@=10pt{
 M \ar@{=>}[rr]^-{H^*(-,d^2)} \ar@{=>}[rd]^>{H^*(-,d^1)} \ar@{=>}[dd]_{H^*(-d^0)} & & H^*(M,d^2) \ar@{=>}'[d][dd]_<<<{d^0-SS} \ar@{=>}[rd]^{d^1-SS} & \\
 & H^*(M,d^1) \ar@{=>}[rr]^<<<<<<{d^2-SS} \ar@{=>}[dd]^<<<{d^0-SS} & & H^*(M,d^{1,2}) \ar@{=>}[dd]^{d^0-SS} \\
 H^*(M,d^0) \ar@{=>}'[r]^<<<<<<{d^2-SS}[rr] \ar@{=>}[rd]_<{d^1-SS} & & H^*(M,d^{0,2}) \ar@{=>}[rd]^{d^1-SS} & \\
 & H^*(M,d^{0,1}) \ar@{=>}[rr]^{d^2-SS} & & H^*(M,d)
 }
\]

We can interpret taking the homology of $M$ with respect to $d^i$ as a degenerate sort of spectral sequence, so that all the edges in the above diagram are spectral sequences.  The resulting ``commutative diagram of spectral sequences'' gives different ways of computing an associated graded of the total cohomology $H^*(M,d)$. Depending on which path we take we are going to get different representatives for classes in $H^*(M,d)$, and different extensions are going to be visible. In fact the cube gives $(n+1)!$ different ways of computing $H^*(M,d)$ up to extensions, and one could hope that in a given situation this is enough to understand $H^*(M,d)$ completely.

This is a rather complicated situation in general. Rather than aiming for a complete understanding of $H^*(M,d)$, we will construct an algorithm that takes as input some $x \in M$ that is a permanent cycle in the sequence of spectral sequences $d^n\mathrm{-}SS, \dotsc, d^0\mathrm{-}SS$, and produces as output a representative for $[x] \in H^*(M,d)$ which is in some sense optimal.

To give the basic idea of our algorithm, consider the double complex $M$ defined by the following figure:
\[ \xymatrix@L=1pt{
 & x & & y & & z \\
a \ar[ur]^{d^1} \ar[urrr]^{d^2} & & b \ar[ur]^-<<<{d^1} \ar[urrr]^{d^2} & & 
} \]
If we first compute $H^*(M, d^2)$ we are left with $x$, and it is clear that $x$ also survives the $d^1$ spectral sequence and represents a nonzero class $[x] \in H^*(M, d)$. But $[x]$ is also represented by $-y$ or by $z$. Our algorithm will give $z$ as a representative for $[x]$.

Here is the algorithm. It goes through $n+1$ steps, starting with step $n$ and ending with step $0$.  We write $[x]_{d^{k,\ldots,n}}$ for the class in $H^*(M; d^{k,\ldots,n})$ determined by $x$.

\begin{defn} \label{d:algo}
Start with an element $x = x_{n + 1}$ in $M$ that is a permanent cycle in the sequence of spectral sequences $d^n\mathrm{-}SS, \dotsc, d^0\mathrm{-}SS$.    We will recursively define elements $x_{k} \in M$ for $0 \leq k \leq n$ so that in step $k$, we take as input a representative $x_{k+1}$ for the class $[x]_{d^{k+1,\ldots,n}}$ and produce as output a representative $x_k$ for the class $[x]_{d^{k,\ldots,n}}$.

For each homogeneous component $x'$ of $x_{k+1}$, check if $x'$ is a $d^k$-boundary. If $x'$ is not a $d^k$-boundary, let $x''=x'$. If $x'$ is a $d^k$-boundary, say, $d^k(a)=x'$, let $x'' = -d^{k+1,\ldots,n}(a)$. Replace $x_{k+1}$ by $\sum x''$. Iterate this procedure until none of the homogeneous components are $d^k$-boundaries.  The resulting element of $M$ is the output $x_{k}$ of the $k$-th step of the algorithm.
\end{defn}

There are two issues with this algorithm. First, there is a choice of the element $a$ satisfying $d^k(a)=x'$ and different choices might give different outputs. Second, without additional assumptions there is no guarantee that this process will terminate in finitely many steps.

In the proof of Theorem \ref{t:main} we deal with the first issue by choosing a particular $a$ (see the proof of Lemma \ref{l:divide}). %We could also consider all possible choices of $a$, and show that for $x$ in total degree $-1$ we always get $x_0=0$. 
The second issue is taken care of by the following result.

\begin{lemma}
Suppose $M^{*,\ldots,*}$ is bounded below in each grading. Then the above process terminates after finitely many steps.
\end{lemma}

\begin{proof}
This follows because each time we replace $x'$ by $x'' = -d^{k+1,\ldots,n}(a)$ we decrease the $k$'th grading by $1$.
\end{proof}

Now let us do an example from ``nature''.

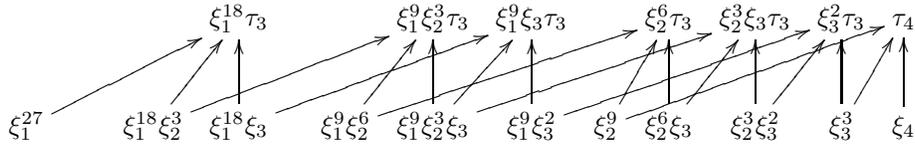
\[ \xymatrix@C=4pt{
 & & & & \xi_1^{18} \tau_3 \ar@{<-}[lllld] \ar@{<-}[ld] \ar@{<-}[d] & & & \xi_1^9 \xi_2^3 \tau_3 \ar@{<-}[lllld] \ar@{<-}[ld] \ar@{<-}[d] & \xi_1^9 \xi_3 \tau_3 \ar@{<-}[lllld] \ar@{<-}[ld] \ar@{<-}[d] & & \xi_2^6 \tau_3 \ar@{<-}[lllld] \ar@{<-}[ld] \ar@{<-}[d] & \xi_2^3 \xi_3 \tau_3 \ar@{<-}[lllld] \ar@{<-}[ld] \ar@{<-}[d] & \xi_3^2 \tau_3 \ar@{<-}[lllld] \ar@{<-}[ld] \ar@{<-}[d] & \tau_4 \ar@{<-}[lllld] \ar@{<-}[ld] \ar@{<-}[d] \\
\xi_1^{27} & & & \xi_1^{18}\xi_2^3 & \xi_1^{18} \xi_3 & & \xi_1^9 \xi_2^6 & \xi_1^9 \xi_2^3 \xi_3 & \xi_1^9 \xi_3^2 & \xi_2^9 & \xi_2^6 \xi_3 & \xi_2^3 \xi_3^2 & \xi_3^3 & \xi_4
} \]

Let $M$ be as above, where each node indicates a free $\bF_3[v_0,v_1,v_2]$-module on that element.  Each arrow of slope $1/4$ represents the $d^2$ differential which is multiplication by $v_2$, each arrow of slope $1$ represents  $d^1$ which is multiplication by $v_1$, and each vertical arrow represents $d^0$ which is multiplication by $v_0$. Consider the class
\[
 x = v_0^2 v_1 \xi_1^{18} \tau_3.
\]
Notice that $x$ is a permanent cycle after running the $d^{2}$ spectral sequence, the $d^{1}$ spectral sequence, then the $d^0$ spectral sequence.  To set up the algorithm, set $x_3 = x$.  We see that $x$ is not a boundary in the $d^2$ spectral sequence, so step 2 amounts to setting $x_2 = x_3$. In step 1, we get the representative
\[
 x_1 = -v_0^2 v_2 \xi_1^9 \xi_2^3 \tau_3
\]
for $[x]_{d^{1,2}}$ by considering $a=v_0^2 \xi_1^{18} \xi_2^3$.

In step 0, we first replace $x_1$ by
\[
 v_0 v_1 v_2 \xi_1^9 \xi_3 \tau_3 + v_0 v_2^2 \xi_2^3 \xi_3 \tau_3,
\]
then we replace each of the two homogeneous summands by $-v_1 v_2^2 \xi_3^2 \tau_3$ to obtain
\[
 x_0 = -2 v_1 v_2^2 \xi_3^2 \tau_3 = v_1 v_2^2 \xi_3^2 \tau_3.
\]
(The last equality follows because we are in characteristic $3$.)  This is not sufficient to conclude that $x$ is a boundary. The class $[x] \in H^*(M,d)$ is probably nonzero but we prefer not to say anything about it.  The point is that while we started with a representative on the ``left hand side'' of $M$, at each step the algorithm gave us a new representative ``further to the right''.  

If instead we start with
\[
 y = v_0^2 v_1 x = v_0^4 v_1^2 \xi_1^{18} \tau_3, 
\]
the algorithm does the following. Step 2 does nothing because $y$ is not a $d^2$ boundary.  In step $1$ we replace $y$ first by $-v_0^4 v_1 v_2 \xi_1^9 \xi_2^3 \tau_3$ and then by
\[
 y_1 = v_0^4 v_2^2 \xi_2^6 \tau_3.
\]
Then in step $0$ we replace $y_1$ first by $-v_0^3 v_1 v_2^2 \xi_2^3 \xi_3 \tau_3$, then by $v_0^2 v_1^2 v_2^2 \xi_3^2 \tau_3$, then by $-v_0 v_1^3 v_2^2 \tau_4$, and finally by $0$. This reflects the fact that there is a differential $d_4^0(\xi_4)=[y]_{d^{1,2}}$ in the $d^0$ spectral sequence.

\subsection{How this fits with the main theorem}
In the proof of Theorem \ref{t:main} we will be interested in similar examples. We wish to show that there is no homology in total (homological) degree $-1$, where the $\xi_i$'s and $\tau_j$'s are in negative degrees and the $v_k$'s are in positive degrees. We will do this by showing that in fact there is nothing in odd total degree $\geq -1$ by bounding the total degree of elements $v_0^{r_0} \cdots v_{n-1}^{r_{n-1}} x$ that can survive the sequence of spectral sequences.

\section{The homology and cohomology of $BP \lan{n}$}
In this section we collect a few facts about the homology and cohomology of $BP \lan{n}$. In particular, we will determine the action of the Milnor primitives $Q_i$ on $H^*(BP \lan{n})$ for $0 \leq i \leq n$. Our method is to first determine the coaction of the dual of $Q_i$ on homology, then dualize.  

Milnor showed \cite{Mi58} that the dual $\cA_*$ of the mod $p$ Steenrod algebra is the Hopf algebra 
\[
 \cA_* = P(\xi_1,\xi_2,\ldots) \otimes E(\tau_{0},\tau_{1},\ldots) \quad \abs{\xi_i} = 2(p^i - 1), \abs{\tau_i} = 2p^i - 1
\]
with coproduct 
\[
 \psi(\xi_k) = \sum_{i = 0}^k \xi_{k - i}^{p^i} \otimes \xi_i \quad \text{and} \quad \psi(\tau_k) = \tau_k \otimes 1 + \sum_{i = 0}^k \xi_{k - i}^{p^i}\otimes \tau_i.
\]
We will use this notation for all primes, with the understanding that if $p=2$ the multiplicative structure is different, with the element $\tau_i$ denoting the element usually called $\xi_{i+1}$ and $\xi_i$ denoting the element usually called $\xi_i^2$. The difference in multiplicative structure at $p=2$ will play no role in our arguments.

It will be useful to use the images $\bar{\xi}_i$, $\bar{\tau}_i$ of Milnor's generators under the conjugation map of the Hopf algebra $\cA_*$.  The coproduct is then given by:
\begin{equation}\label{eq:conjugate_comult}
\psi(\bar{\xi}_k) = \sum_{i = 0}^k \bar{\xi}_i \otimes \bar{\xi}_{k - i}^{p^i} \quad \text{and} \quad \psi(\bar{\tau}_k) = 1 \otimes \bar{\tau}_k + \sum_{i = 0}^k \bar{\tau}_i \otimes  \bar{\xi}_{k - i}^{p^i}.
\end{equation}

The Milnor primitives $Q_i \in \cA$ are inductively defined by $Q_0 = \beta$ (the mod $p$ Bockstein) and
\[
 Q_i = P^{p^i} Q_{i-1} - Q_{i-1} P^{p^i}.
\]
Using the Milnor basis, the element $Q_i \in \cA$ is dual to $\tau_i \in \cA_*$ and has homological degree $-(2p^i - 1)$.

The mod $p$ cohomology of $BP\lan{n}$ as a left module over the Steenrod algebra is 
\[
 H^*(BP \lan{n}) = \cA//E_n,
\]
where $E_n=E(Q_0,\ldots,Q_n)$ is the exterior algebra on the first $n + 1$ Milnor primitives \cite[Prop.\ 1.7]{Wi75}. Here we quotient out the exterior algebra on the \emph{right}, while we are interested in the \emph{left} action of $E_n$ on the result. Because the Steenrod algebra is non-commutative the left action is still highly nontrivial.

Dually, the mod $p$ homology of $BP\lan{n}$ is the algebra
\[
 H_*(BP\lan{n}) = P(\bar{\xi}_1,\bar{\xi}_2,\ldots) \otimes E(\bar{\tau}_{n+1},\bar{\tau}_{n+2},\ldots).
\]
The left $\cA_*$ comodule structure map $\psi_{L} \colon H_*(BP\lan{n}) \to \cA_* \otimes H_*(BP\lan{n})$ is given on the conjugate generators by formula \eqref{eq:conjugate_comult} and is then extended multiplicatively to all of $H_*(BP\lan{n})$.

Dualizing, we derive the left action of $\alpha \in \cA$ on $f \in H^*(BP\lan{n})$ by the formula
\begin{align*}
(\alpha \cdot f)(x) =  \lan{ \alpha \otimes f, \psi_L(x)} = \sum_i (-1)^{\abs{f}\abs{a_i}} \lan{\alpha, a_i}\lan{f, x_i}
\end{align*}
where $\lan{-, -}$ denotes the pairing of a module and its dual and $\psi_L(x) = \sum_i a_i \otimes x_i$ is the left coaction of $\cA_*$ on $x \in H_*(BP \lan{n})$.  For $x \in \cA_*$, let $x^* \in \cA$ denote the dual. Using this description, it is straightforward to calculate that $Q_i \cdot (\bar{\xi}_{j}^{p^i})^* = (\bar{\tau}_{j + i})^*$ and that $Q_i$ acts as zero on smaller powers of $\xi_j$, as well as $\tau_j$.  From here we can extend multiplicatively to deduce the following lemma.
\begin{lemma}\label{l:Q_action}
Let $x^* = (\bar{\xi}^{e_1}_{j_1} \dotsm \bar{\xi}^{e_m}_{j_m}\cdot \bar{\tau}_J)^* \in \cA$ be the dual of a monomial in the $\bar{\xi}_j$ and the $\bar{\tau}_{j}$.  Then $Q_i$ acts on $x^*$ in the following manner.  For each $k$ such that $x$ contains a factor of $\bar{\xi}_{j_k}^{p^i}$, replace that factor by $\bar{\tau}_{j_k + i}$ and dualize, then take the sum over all such expressions.  In other words we have:
\[
Q_i \cdot x^* = \sum_{\substack{k \; \text{\emph{such that}} \\ p^i \leq e_k}} \Bigl(  \bar{\xi}^{e_k - p^i}_{j_k} \bar{\tau}_{j_k + i} \cdot \prod_{ \ell \neq k} \bar{\xi}_{j_\ell}^{e_\ell} \cdot \bar{\tau}_J \Bigr)^*
\]
\end{lemma}
 
\noindent For example, at $p=3$ we have
\[ 
 Q_1\cdot (\bar{\xi}_2^6 \bar{\xi}_3^3)^* = (\bar{\xi}_2^3 \bar{\xi}_3^3 \bar{\tau}_3)^* + (\bar{\xi}_2^6 \bar{\tau}_4)^*.
\]

Our convention is to use homological grading, so $H^*(BP \lan{n})$ is concentrated in non-positive degree. Hence $\bar{\xi}_j^*$ is in degree $-(2p^j-2)$ and $\bar{\tau}_j^*$ is in degree $-(2p^j-1)$.

\begin{lemma} \label{l:divide}
Fix $0 \leq i \leq n$ and suppose $x^* \in H^*(BP \lan{n})$ is in odd degree with $Q_i \cdot x^*=0$. Then there is a not necessarily unique $y^* \in H^*(BP \lan{n})$ with $Q_i \cdot y^*=x^*$.
\end{lemma}

\begin{proof}
Order the monomials in $x \in H_*(BP \lan{n})$ lexicographically by looking at the $\bar{\tau}_j$, starting with the largest $j$, and then by the $\bar{\xi}_i$, starting with the largest $i$.

Using this ordering, let $x_1$ be the largest monomial in $x$ and suppose $\ttau_a$ is the largest $\ttau_j$ that divides $x_1$. (Because $x$ is in odd degree some such $\ttau_a$ must exists.) Note that the summand $x_1$ cannot contain any $\bar{\xi}_b^{p^i}$ for $b+i>a$ because then $Q_i \cdot x_1^*$ would contain a $(x_1/\xxi_b^{p^i} \cdot \ttau_{b+i})^*$ and this cannot cancel with $Q_i \cdot (x^*-x_1^*)$ because every term of $Q_i \cdot (x^*-x_1^*)$ is smaller.

Now let $y_1=x_1/\ttau_a \cdot \xxi_{a-i}^{p^i}$. Then $Q_i \cdot y_1^* = x_1^* + (x_1')^*$ where $x_1'$ is smaller. Replace $x^*$ by $x^*-Q_i \cdot y_1^*$ and proceed by induction.
\end{proof}

\section{The Adams spectral sequence}
Under the hypotheses of Theorem \ref{t:main} there is a conditionally convergent Adams spectral sequence
\[
 E_2^{s,t} = Ext_\cA^{s,t}(H^*(BP\lan{n}), H^*(X)) \Longrightarrow \pi_{t-s} Hom(X,BP\lan{n}\pc).
\]

\noindent By a change-of-rings isomorphism \cite[A1.3.12]{Ra86}, the $E_2$ term of the Adams spectral sequence is given by
\[ E_2^{s,t} = Ext_\cA^{s,t}(\cA // E_n, H^*(X)) \cong Ext_{E_n}^{s,t}(\bF_p, H^*(X)).\]
Using the given isomorphism $\theta$, we make the identification $H^*(X) \cong \cA//E_n$, with $E_{n}$ acting on the left. The $\bF_p$ dual of the Koszul resolution of $\bF_p$ as an $E_n$-module is the chain complex $P(v_0, \dotsc, v_n) \otimes E_n$ graded by the degree of homogeneous polynomial generators and with differential 
\[
d(p \otimes \omega) = \sum_i v_i p \otimes (Q_i \omega).
\]
Thus the $E_2$ term of the Adams spectral sequence is isomorphic to the homology of the chain complex 
\[ M = H^*(BP \lan{n})[v_0,\ldots,v_n], \]
where the differential is given by
\[ d(x^*) = \sum_{i=0}^n Q_i(x^*) v_i. \]
Notice that the element $v_i$ is in bidegree $(1, 2p^i - 1)$.  We are now in the situation discussed in Section \ref{s:genSS}.

Since $H^*(BP\lan{n})$ is finite in each degree, this description of the $E_2$ page shows us that the group $E_{2}^{s, t}$ is finite in each bidegree.  This implies that Boardman's derived $E_{\infty}$ term $R E_{\infty} = 0$ and so the spectral sequence converges strongly \cite[Theorem 7.1]{Bo99}.  In particular, if the given map $\theta : H^*(BP\lan{n}) \to H^*(X)$, considered as a class $[\theta] \in E_2^{0,0}$, survives to $E_\infty^{0,0}$, then $[\theta]$ represents a map $f \colon X \to BP\lan{n}\pc$ satisfying $H^*(f)=\theta \circ H^*(i)$, where $i \colon BP\lan{n} \to BP\lan{n}\pc$ is the $p$-completion map.  In order to prove that the given isomorphism of cohomology $\theta$ survives the spectral sequence, it suffices to show that there is nothing in total degree $-1$ in $E_2^{*,*}$ and thus no room for nontrivial differentials on $\theta$.  We will prove this in the next section, completing the proof of Theorem \ref{t:main}

\begin{remark}\label{remark:convergence}
We believe that a $p$-local version of the main theorem is true.  The only obstruction is the lack of convergence of the Adams spectral sequence 
\[
Ext_\cA^{s,t}(H^*(BP\lan{n}), H^*(X)) \Longrightarrow \pi_{t-s} Hom(X,BP\lan{n})
\]
before passage to the $p$-completion of $BP\lan{n}$ in the abutment.  Adams and Priddy get around this using a clever argument involving Adams operations which is not currently available in the general case of $BP\lan{n}$.
\end{remark}

\section{Proof of the main theorem}

We need one more ingredient. As an $\cA$-module $H^*(BP \lan{n})$ is irreducible, but when considering $H^*(BP \lan{n})$ as an $E_n$-module it splits as a direct sum of submodules, each one finite-dimensional over $\bF_p$. We make the following definition.

\begin{defn}
We define a new multiplicative grading on $H_*(BP \lan{n})$ called the weight by declaring that $\xxi_j$ and $\ttau_j$ have weight $2p^j$.
\end{defn}

The degree $\delta$ of an element is slightly less than the weight $w$, and satisfies the inequality
\[ \frac{p-1}{p} w \leq \delta \leq w-1.\]
We use $\delta$ for the degree rather than $d$, as $d$ is already overloaded. Let $H^*(BP \lan{n})[w]$ denote the dual of the weight $w$ part of $H_*(BP \lan{n})$.  By Lemma \ref{l:Q_action}, the action of $Q_i$ preserves the subspace $H^*(BP \lan{n})[w]$.  Hence
\[
 Ext^{s,t}_{E_n}(\bF_p, H^*(BP \lan{n}))
\]
splits as a direct sum
\[
 \bigoplus_w Ext^{s,t}_{E_n}(\bF_p, H^*(BP \lan{n})[w]).
\]

\begin{lemma}
Suppose $0 \leq i \leq n$. There are no infinite $v_i$-towers in odd total degree in $Ext_{E_n}^{*,*}(\bF_p, H^*(BP \lan{n})$. 
\end{lemma}

\begin{proof}
Lemma \ref{l:divide} shows that if we calculate $Ext_{E_n}^{*,*}(\bF_p, H^*(BP \lan{n}))$ by first taking homology in the $i$'th direction there are no $v_i$-towers in odd degree in the associated graded. The above splitting of $Ext$ as a direct sum of $Ext_{E_n}^{*,*}(\bF_p, H^*(BP\lan{n})[w])$ shows that there is no room for extensions conspiring to make infinite $v_i$-towers.
\end{proof}

Because $H^*(BP \lan{n})$ is concentrated in non-positive degree, with the first odd-degree element $\ttau_{n+1}$ in degree $-(2p^{n+1}-1)$, the proof of Theorem \ref{t:main} amounts to a better understanding of the above lemma.

\begin{proof}[Proof of Theorem \ref{t:main}]
It suffices to prove that there is nothing in total degree $-1$ one weight at a time. Fix a weight $w$, necessarily divisible by $2p$. The lowest degree element of weight $w$ in $H_*(BP \lan{n})$ is
\[
 \xxi_1^{w/2p} \quad \text{in degree} \quad \frac{p-1}{p} w
\]
and the lowest odd degree element of weight $w$ is, assuming $w$ is large enough for such an element to exist,
\[
 \xxi_1^{w/2p-p^n} \ttau_{n+1} \quad \text{in degree} \quad \frac{p-1}{p} w +2p^n-1.
\]

Now consider some nonzero class $\alpha \in E_2^{*,*}[w] \subset H_*(H^*(BP \lan{n})[v_0,\ldots,v_n], d)$, and assume without loss of generality that $\alpha$ is a homogeneous element. By Lemma \ref{l:divide}, we can represent $\alpha$ by an element of the form
\[ \alpha_n = v_0^{r_0} \cdots v_{n-1}^{r_{n-1}} x_n^* \]
with $x_n \in H_*(BP \lan{n})$ in weight $w$ and degree $\delta_n$ (if $\alpha_n$ contained a positive $v_n$-power then $\alpha_n$ would be a $d^n$-boundary).  In other words, $\alpha_n$ is a $d^n$-cycle but not a $d^n$-boundary, and its class in $d^{n}$-homology survives the sequence of spectral sequences starting with the $d^{n-1}$ spectral sequence and ending with the $d^0$ spectral sequence to represent $\alpha$ in the associated graded.  Notice that $x_n^*$ is in  degree $-\delta_n$, so the class $\alpha = [\alpha_n]$ is in total degree
\begin{equation}\label{eq:deg_alpha}
\deg(\alpha) = r_1(2p-2) + r_2(2p^2-2)+\dotsm + r_{n-1}(2p^{n-1}-2)-\delta_n.
\end{equation}

Now we use the algorithm from Definition \ref{d:algo}. Since $\alpha_n$ is not a $d^{n}$-boundary, step $n$  returns $\alpha_n$ as a representative for the class $[\alpha]$ in the $d^{n}$-homology of $H^*(BP \lan{n})[v_0,\ldots,v_n]$.

Let's examine step $n-1$ for the example $\alpha_n = v_{n-1}^2 x^*$.  We replace $\alpha_n$ by the element $\alpha_{n-1} = v_n^2 z^*$ defined by the following picture.
\[ \xymatrix@=15pt@L=1pt{
 & \alpha_n = v_{n-1}^2 x^* & & v_{n-1} v_n y^* & & \alpha_{n-1} = v_n^2 z^* \\
v_{n-1} a^* \ar[ur]^{d^{n-1}} \ar[urrr]^{d^n} & & v_n b^* \ar[ur]^-<{d^{n-1}} \ar[urrr]^{d^n} & & 
} \]
Returning to the general sitution, Lemma \ref{l:divide} ensures that there is always such a zig-zag, trading the factor of $v_{n-1}^{r_{n-1}}$ in $\alpha_n$ for a $v_n^{r_{n-1}}$ in $\alpha_{n-1}$. Hence
\[
 \alpha_{n-1} = v_0^{r_0} \cdots v_{n-2}^{r_{n-2}} v_n^{r_{n-1}} x_{n-1}^*,
\]
for some class $x_{n-1}$ in degree
\[
 \delta_{n-1} = \delta_n + r_{n-1}(2p^n-2p^{n-1}).
\]

Similarly, in step $n-2$ we get a new representative $\alpha_{n-2}$ by iteratively replacing $v_{n-2} x^*$ with $- v_{n-1} y^* - v_n z^*$ as in the picture

\[ \xymatrix@=15pt@L=1pt{
\alpha_{n-1} = v_{n-2} x^* & v_{n-1} y^* & & & v_n z^* \\
a^* \ar[u] \ar[ur] \ar[urrrr] & & & & 
} \]
Thus $\alpha_{n-2}$ is a sum of terms of the form
\[
 v_0^{r_0} \dotsm v_{n - 3}^{r_{n - 3}} \sum_{j = 0}^{r_{n - 2}} v_{n - 1}^{r_{n - 2} - j} v_{n}^{j + r_{n - 1}} x_{n-2,j}^*.
\]
For each $j$ in the sum, $\deg(x_{n-2,j}) \geq \delta_{n-2}$ where
\[
 \delta_{n-2} = \delta_{n-1} + r_{n-2}(2p^{n-1}-2p^{n-2}).
\]

Continuing in this way, we end up with a representative
\[
 \alpha_0 = \sum_{j} v_{I_j} x_{0,j}^*
\]
for $\alpha$, where $v_{I_j}$ is a monomial in the $v_{i}$ with $i \geq 1$.  Also each $\deg(x_{0,j}) \geq \delta_0$, where $\delta_0$ is given by
\begin{equation}\label{eq_delta0}
 \delta_0 =  \delta_n + r_0(2p-2) + r_1(2p^2-2p)+\ldots+r_{n-1}(2p^n-2p^{n-1}).
 \end{equation}
The bound on the degree of $x_{0, j}$ is a consequence of the fact that each time we trade in a $v_i$ we increase the degree of the elements in $H_*(BP \lan{n})$ by at least $2p^{i+1}-2p^i$.

Because $x_n$ is in odd degree $\delta_n$ with weight $w$, we have:
\[
\delta_{n} \geq \frac{p - 1}{p}w + 2p^{n} - 1.
\]
Therefore, equation \eqref{eq:deg_alpha} implies that:
\[ 
 r_1(2p-2)+\dotsm+r_{n-1}(2p^{n-1}-2) \geq \frac{p-1}{p} w + 2p^{n} - 1 + \deg(\alpha).
\]
Since $2p^{i+1}-2p^i > 2p^{i} - 2$, equation \eqref{eq_delta0} gives:
\begin{align*}
\delta_0 &\geq \delta_n + r_0(2p - 2) + \frac{p - 1}{p}w + 2p^n - 1 + \deg(\alpha) \\
&\geq \frac{p-1}{p} w + 2p^n-1 + \frac{p-1}{p} w + 2p^n - 1 + \deg(\alpha) \\
&\geq  w + 4p^n - 2 + \deg(\alpha)
\end{align*}
The weight and degree of an element of $H_*(BP\lan{n})$ must satisfy $w \geq \delta + 1$, so we deduce that
\[
\deg(\alpha) \leq 1 - 4p^n.
\]
In particular, $\alpha$ cannot lie in total degree $-1$.  This finishes the proof of the main theorem.

\end{proof}

\bibliographystyle{plain}
\bibliography{b}

\begin{thebibliography}{10}

\bibitem{AdPr76}
J.~F. Adams and S.~B. Priddy.
\newblock Uniqueness of {$B{\rm SO}$}.
\newblock {\em Math. Proc. Cambridge Philos. Soc.}, 80(3):475--509, 1976.

\bibitem{Ar73}
Sh{\^o}r{\^o} Araki.
\newblock {\em Typical formal groups in complex cobordism and {$K$}-theory}.
\newblock Kinokuniya Book-Store Co., Ltd., Tokyo, 1973.
\newblock Lectures in Mathematics, Department of Mathematics, Kyoto University,
  No. 6.

\bibitem{Bo99}
J.~Michael Boardman.
\newblock Conditionally convergent spectral sequences.
\newblock In {\em Homotopy invariant algebraic structures ({B}altimore, {MD},
  1998)}, volume 239 of {\em Contemp. Math.}, pages 49--84. Amer. Math. Soc.,
  Providence, RI, 1999.

\bibitem{Bo79}
A.~K. Bousfield.
\newblock The localization of spectra with respect to homology.
\newblock {\em Topology}, 18(4):257--281, 1979.

\bibitem{Ha76}
Michiel Hazewinkel.
\newblock Constructing formal groups. {I}. {T}he local one dimensional case.
\newblock {\em J. Pure Appl. Algebra}, 9(2):131--149, 1976/77.

\bibitem{LaNa11}
Tyler Lawson and Niko Naumann.
\newblock Commutativity conditions for truncated {B}rown-{P}eterson spectra of
  height 2.
\newblock {\em Preprint, available on the arXiv}, 2011.

\bibitem{Mi58}
John Milnor.
\newblock The {S}teenrod algebra and its dual.
\newblock {\em Ann. of Math. (2)}, 67:150--171, 1958.

\bibitem{Qu69}
Daniel Quillen.
\newblock On the formal group laws of unoriented and complex cobordism theory.
\newblock {\em Bull. Amer. Math. Soc.}, 75:1293--1298, 1969.

\bibitem{Ra86}
Douglas~C. Ravenel.
\newblock {\em Complex cobordism and stable homotopy groups of spheres}, volume
  121 of {\em Pure and Applied Mathematics}.
\newblock Academic Press, Inc., Orlando, FL, 1986.

\bibitem{Wi73}
W.~Stephen Wilson.
\newblock The {$\Omega $}-spectrum for {B}rown-{P}eterson cohomology. {I}.
\newblock {\em Comment. Math. Helv.}, 48:45--55; corrigendum, ibid. 48 (1973),
  194, 1973.

\bibitem{Wi75}
W.~Stephen Wilson.
\newblock The {$\Omega $}-spectrum for {B}rown-{P}eterson cohomology. {II}.
\newblock {\em Amer. J. Math.}, 97:101--123, 1975.

\end{thebibliography}

\end{document}